\documentclass[11pt,a4paper,twoside,onecolumn, english]{extarticle}
\usepackage[utf8]{inputenc}
\usepackage{color}
\usepackage{float}
\usepackage{amsmath}
\usepackage{amsthm}
\usepackage{amssymb}
\usepackage{graphicx}
\usepackage[numbers]{natbib}
\usepackage[unicode=true,
 bookmarks=false,
 breaklinks=false,pdfborder={0 0 1},backref=section,colorlinks=false]
 {hyperref}

\makeatletter
\theoremstyle{plain}
\newtheorem{thm}{\protect\theoremname}
\theoremstyle{plain}
\newtheorem{prop}[thm]{\protect\propositionname}

\usepackage{stmaryrd}
\makeatother

\providecommand{\propositionname}{Proposition}
\providecommand{\theoremname}{Theorem}

\begin{document}
\title{Controlling independently the heading and the position of a Dubins
car using Lie brackets}
\author{Luc Jaulin}
\maketitle
\begin{abstract}
In this paper, we give a control approach to follow a trajectory for
a Dubins car controlling the heading independently. The difficulty
is that the Dubins car should have a heading corresponding to the
argument of the vector speed of the vehicle. This non holonomic constraint
can be relaxed thanks to a specific control which uses Lie brackets.
We will show such a control will allow us, at least from a theoretical
point of view, to control independently the moving position of the
car and the heading. 
\end{abstract}

\section{Introduction}

The motivation of Lie brackets for control is illustrated by Figure
\ref{fig:lie_backet_car} on the car parking problem. We would like
to move the blue car sideway between the two red cars already parked.
Moving sideway is not possible, but we can approach this direction
using many small forward-backward maneuvers. Building this new sideway
direction corresponds to the Lie bracket operator between two dynamics. 

The difficulty that we meet when we park a car is that we have two
inputs (the acceleration and the steering wheel) whereas we want to
control three quantities: two for the position and one for the heading.
The goal of this paper is to illustrate how controlling three directions
from two actuators can be done using Lie brackets. This will be illustrated
on the case of a first order Dubins car and a second order Dubins
car \citep{Dubins57}\citep{laumond:86}. 

\begin{figure}[H]
\centering\includegraphics[width=14cm]{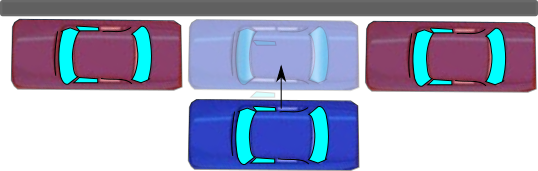}
\caption{Lie brackets can help to maneuver in order to park the blue car}
\label{fig:lie_backet_car}
\end{figure}

The paper is organized as follows. Section \ref{sec:Control_driftless}
shows how to control driftless systems using Lie brackets. Section
\ref{sec:first:order:dubins} illustrates the principle on the first
order Dubins car and Section \ref{sec:second:order:dubins} on the
second order Dubins car. Section \ref{sec:Conclusion} concludes the
paper.

\section{Control of a driftless system with two inputs\label{sec:Control_driftless}}

A \emph{driftless system} is a system of the form
\begin{equation}
\dot{\mathbf{x}}=\mathbf{F}(\mathbf{x})\cdot\mathbf{u}.
\end{equation}
where $\mathbf{x}\in\mathbb{R}^{n}$ is the state vector, $\mathbf{u}\in\mathbb{R}^{m}$
is the input vector, $\mathbf{F}$ is a $n\times m$ matrix which
depends on $\mathbf{x}$. If $\mathbf{u}=\mathbf{0}$ then the evolution
of the system stops. Lie brackets are mainly used to control this
type of systems. In this section, for simplicity, we consider that
$m=2$, which means that we have two inputs.

\subsection{Definition of Lie brackets}

\noindent Consider two vector fields \textbf{f} and $\mathbf{g}$
of $\mathbb{R}^{n}$ corresponding to two state equations $\dot{\mathbf{x}}=\mathbf{f}(\mathbf{x})$
and $\dot{\mathbf{x}}=\mathbf{g}(\mathbf{x})$. We define the Lie
bracket between these two vector fields as
\begin{equation}
[\mathbf{f},\mathbf{g}]=\frac{d\mathbf{g}}{d\mathbf{x}}\cdot\mathbf{f}-\frac{d\mathbf{f}}{d\mathbf{x}}\cdot\mathbf{g}.
\end{equation}
The adjoint notation can be used: 
\begin{equation}
\text{ad}_{\mathbf{f}}\mathbf{g}=[\mathbf{f},\mathbf{g}].
\end{equation}

We can check that the set of vector fields equipped with the Lie bracket
is a Lie algebra. The proof is not difficult but tedious. As a consequence
we have the following Jacobi equality
\begin{equation}
[\mathbf{f},[\mathbf{g},\mathbf{h}]]+[\mathbf{h},[\mathbf{f},\mathbf{g}]]+[\mathbf{g},[\mathbf{h},\mathbf{f}]]=\mathbf{0}.
\end{equation}
 In this paper, the fact that the set of vector fields equipped with
the Lie bracket is a Lie algebra will not be needed.

As an example, consider the two linear vector fields $\mathbf{f}(\mathbf{x})=\mathbf{A}\cdot\mathbf{x}$,
$\mathbf{g}(\mathbf{x})=\mathbf{B}\cdot\mathbf{x}$. We have
\begin{equation}
\begin{array}{ccl}
[\mathbf{f},\mathbf{g}]\left(\mathbf{x}\right) & = & \frac{d\mathbf{g}}{d\mathbf{x}}\cdot\mathbf{f}\left(\mathbf{x}\right)-\frac{d\mathbf{f}}{d\mathbf{x}}\cdot\mathbf{g}\left(\mathbf{x}\right)\\
 & = & \mathbf{B}\cdot\mathbf{A}\cdot\mathbf{x}-\mathbf{A}\cdot\mathbf{B}\cdot\mathbf{x}\\
 & = & \left(\mathbf{B}\mathbf{A}-\mathbf{A}\mathbf{B}\right)\cdot\mathbf{x}.
\end{array}
\end{equation}

\subsection{Creating new direction with Lie brackets}

Lie brackets are often used to check the local accessibility near
driftless states \citep{Pernebo81} \citep{AndreaNovel88}. The accessibility
distribution (or \emph{Lie ideal}) Lie$(\mathbf{f},\mathbf{g})$ of
two vector fields $\mathbf{f},\mathbf{g}$ is obtained by taking all
vector fields that can be generated using Lie brackets from $\mathbf{f}$
and $\mathbf{g}$ \citep{DulebaK05}. If at a driftless state $\bar{\mathbf{x}}$,
Lie$(\mathbf{f},\mathbf{g})$ spans all directions of $\mathbb{R}^{n}$,
then we can conclude \citep{LaValle:06}) that the system is locally
accessible. Now, Lie brackets can also be used to control driftless
systems as shown by the following proposition.
\begin{prop}
Consider the driftless system
\begin{equation}
\dot{\mathbf{x}}=\mathbf{f}(\mathbf{x})\cdot u_{1}+\mathbf{g}(\mathbf{x})\cdot u_{2}.
\end{equation}
Consider the following cyclic sequence:
\[
\begin{array}{ccccccccc}
t\in[0,\delta] &  & t\in[\delta,2\delta] &  & t\in[2\delta,3\delta] &  & t\in[3\delta,4\delta] & t\in[4\delta,5\delta] & \dots\\
\mathbf{u}=\left(1,0\right) &  & \mathbf{u}=\left(0,1\right) &  & \mathbf{u}=\left(-1,0\right) &  & \mathbf{u}=\left(0,-1\right) & \mathbf{u}=\left(1,0\right) & \dots
\end{array}
\]
where $\delta$ is an infinitesimal time period $\delta$. This periodic
sequence will be denoted by
\begin{equation}
\left\{ \left(1,0\right),\left(0,1\right),\left(-1,0\right),\left(0,-1\right)\right\} 
\end{equation}
We have: 
\begin{equation}
\mathbf{x}(t+2\delta)=\mathbf{x}\left(t-2\delta\right)+[\mathbf{f},\mathbf{g}]\left(\mathbf{x}(t)\right)\delta^{2}+o\left(\delta^{2}\right).
\end{equation}
\end{prop}

\begin{proof}
Without loss of generality, the proof will be done for $t=0$. We
shall use the following notations
\begin{equation}
\begin{array}{ccc}
\frac{d\mathbf{f}}{d\mathbf{x}}(\mathbf{x}(t)) & \leftrightarrow & \mathbf{A}_{t}\\
\frac{d\mathbf{g}}{d\mathbf{x}}(\mathbf{x}(t)) & \leftrightarrow & \mathbf{B}_{t}
\end{array}
\end{equation}
For a given $t$, and a small $\delta$ we have
\begin{equation}
\mathbf{x}(t+\delta)-\mathbf{x}\left(t\right)=\dot{\mathbf{x}}(t)\cdot\delta+\ddot{\mathbf{x}}(t)\cdot\frac{\delta^{2}}{2}+o\left(\delta^{2}\right)\label{eq:diff:x:delta}
\end{equation}
with
\begin{equation}
\begin{array}{cccc}
\dot{\mathbf{x}}(t) & = & \mathbf{f}\left(\mathbf{x}\left(t\right)\right)u_{1}+\mathbf{g}\left(\mathbf{x}\left(t\right)\right)u_{2}\\
\ddot{\mathbf{x}}(t) & = & \left(\mathbf{A}_{t}\cdot u_{1}+\mathbf{B}_{t}\cdot u_{2}\right)\dot{\mathbf{x}}(t) & \text{\,\,\,(if \ensuremath{\mathbf{u}} is constant)}
\end{array}
\end{equation}
Taking the cyclic sequence, we have
\begin{equation}
\begin{array}{cclcc}
\dot{\mathbf{x}}(0) & = & -\mathbf{f}\left(\mathbf{x}\left(0\right)\right)\\
\ddot{\mathbf{x}}(0) & = & -\mathbf{A}_{0}\dot{\mathbf{x}}(0)=\mathbf{A}_{0}\mathbf{f}\left(\mathbf{x}\left(0\right)\right) & \,\,\, & (u_{1}=-1,u_{2}=0)\\
\dot{\mathbf{x}}(\delta) & = & -\mathbf{g}\left(\mathbf{x}\left(\delta\right)\right)\\
\ddot{\mathbf{x}}(\delta) & = & -\mathbf{B}_{\delta}\dot{\mathbf{x}}(\delta)=\mathbf{B}_{\delta}\mathbf{g}\left(\mathbf{x}\left(\delta\right)\right) &  & (u_{1}=0,u_{2}=-1)
\end{array}
\end{equation}
We get 
\begin{equation}
\begin{array}{cccc}
\mathbf{x}(\delta)-\mathbf{x}\left(0\right) & \overset{(\ref{eq:diff:x:delta})}{=} & -\mathbf{f}(\mathbf{x}(0))\delta+\mathbf{A}_{0}\cdot\mathbf{f}(\mathbf{x}(0))\cdot\frac{\delta^{2}}{2}+o\left(\delta^{2}\right) & \,\\
\mathbf{x}(2\delta)-\mathbf{x}\left(\delta\right) & \overset{(\ref{eq:diff:x:delta})}{=} & -\mathbf{g}(\mathbf{x}(\delta))\delta+\mathbf{B}_{\delta}\cdot\mathbf{g}(\mathbf{x}(\delta))\cdot\frac{\delta^{2}}{2}+o\left(\delta^{2}\right)
\end{array}
\end{equation}
The sum yields
\[
\begin{array}{ccc}
\mathbf{x}(2\delta)-\mathbf{x}\left(0\right) & = & -\mathbf{f}(\mathbf{x}(0))\cdot\delta+\mathbf{A}_{0}\cdot\mathbf{f}(\mathbf{x}(0))\cdot\frac{\delta^{2}}{2}-\mathbf{g}(\mathbf{x}(\delta))\cdot\delta+\mathbf{B}_{\delta}\cdot\mathbf{g}(\mathbf{x}(\delta))\cdot\frac{\delta^{2}}{2}+o\left(\delta^{2}\right).\end{array}
\]
Now, 
\[
\begin{array}{ccccc}
(i) &  & \mathbf{g}(\mathbf{x}(\delta)) & = & \mathbf{g}(\mathbf{x}(0))+\mathbf{B}_{0}.\underset{=-\mathbf{f}(\mathbf{x}(0))}{\underbrace{\dot{\mathbf{x}}(0)}}\cdot\delta+o\left(\delta\right)=\mathbf{g}(\mathbf{x}(0))+o\left(1\right)\\
(ii) &  & \mathbf{B}_{\delta} & = & \mathbf{B}_{0}+o(1)
\end{array}
\]
Therefore
\begin{equation}
\begin{array}{ccc}
\mathbf{x}(2\delta)-\mathbf{x}\left(0\right) & = & -\mathbf{f}(\mathbf{x}(0))\cdot\delta+\mathbf{A}_{0}\cdot\mathbf{f}(\mathbf{x}(0))\cdot\frac{\delta^{2}}{2}\\
 &  & -\underset{\overset{(i)}{=}\mathbf{g}(\mathbf{x}(0))-\mathbf{B}_{0}.\mathbf{f}(\mathbf{x}(0))\cdot\delta+o\left(\delta\right)}{\underbrace{\mathbf{g}(\mathbf{x}(\delta))}}\cdot\delta+\underset{\overset{(ii)}{=}\mathbf{B}_{0}+o(1)}{\underbrace{\mathbf{B}_{\delta}}}\cdot\underset{\overset{(i)}{=}\mathbf{g}(\mathbf{x}(0))+o\left(1\right)}{\underbrace{\mathbf{g}(\mathbf{x}(\delta))}}\cdot\frac{\delta^{2}}{2}+o\left(\delta^{2}\right)\\
 & = & -\mathbf{f}(\mathbf{x}(0))\cdot\delta+\mathbf{A}_{0}\cdot\mathbf{f}(\mathbf{x}(0))\cdot\frac{\delta^{2}}{2}\\
 &  & -\mathbf{g}(\mathbf{x}(0))\cdot\delta+\mathbf{B}_{0}\cdot\mathbf{f}(\mathbf{x}(0))\cdot\delta^{2}+\mathbf{B}_{0}\cdot\mathbf{g}(\mathbf{x}(0))\cdot\frac{\delta^{2}}{2}+o\left(\delta^{2}\right)
\end{array}\label{eq:x2delta}
\end{equation}
The same reasoning yields 
\[
\begin{array}{ccc}
\mathbf{x}(-2\delta)-\mathbf{x}\left(0\right) & = & -\mathbf{g}(\mathbf{x}(0))\cdot\delta+\mathbf{B}_{0}\cdot\mathbf{g}(\mathbf{x}(0))\cdot\frac{\delta^{2}}{2}\\
 &  & -\mathbf{f}(\mathbf{x}(0))\cdot\delta+\mathbf{A}_{0}.\mathbf{g}(\mathbf{x}(0))\cdot\delta^{2}+\mathbf{A}_{0}\cdot\mathbf{f}(\mathbf{x}(0))\cdot\frac{\delta^{2}}{2}+o\left(\delta^{2}\right)
\end{array}
\]
This result could be obtained directly from (\ref{eq:x2delta}) by
rewriting: $\delta\rightarrow-\delta,\mathbf{f}\rightarrow-\mathbf{g},\mathbf{g}\rightarrow-\mathbf{f},\mathbf{A}_{0}\rightarrow-\mathbf{B}_{0},\mathbf{B}_{0}\rightarrow-\mathbf{A}_{0}$
. Thus
\begin{equation}
\begin{array}{ccc}
\mathbf{x}(2\delta)-\mathbf{x}(-2\delta) & = & \underset{[\mathbf{f},\mathbf{g}]\left(\mathbf{x}(0)\right)}{\underbrace{\left(\mathbf{B}_{0}\cdot\mathbf{f}(\mathbf{x}(0))-\mathbf{A}_{0}\cdot\mathbf{g}(\mathbf{x}(0))\right)}}\delta^{2}+o\left(\delta^{2}\right)\end{array}
\end{equation}
The consequence is that using the periodic sequence, we are now able
to move with respect to the direction $[\mathbf{f},\mathbf{g}]$.
\end{proof}

\subsection{Controlling the new direction created by the Lie brackets}

In this section, we want to find which periodic sequence we should
take in order to follow the field $\nu\cdot[\mathbf{f},\mathbf{g}]$.
We will first consider the case $\nu\geq0$ and then the case $\nu\leq0$. 

We have shown that within a time period of $4\delta$ we moved in
the direction $[\mathbf{f},\mathbf{g}]$ by $[\mathbf{f},\mathbf{g}]\delta^{2}$.
This means that we follow the field $\frac{\delta^{2}}{4\delta}[\mathbf{f},\mathbf{g}]=\frac{\delta}{4}[\mathbf{f},\mathbf{g}]$
which is infinitesimal. Multiplying the cyclic sequence by the scalar
$\alpha\in\mathbb{R}$ amounts to multiplying both $\mathbf{f},\mathbf{g}$
by $\alpha$. The field thus becomes $\frac{\delta}{4}[\alpha\mathbf{f},\alpha\mathbf{g}]=\frac{\delta}{4}\alpha^{2}[\mathbf{f},\mathbf{g}]$.
If we want the field $\nu\cdot[\mathbf{f},\mathbf{g}]$, we have to
multiply the sequence by $\alpha=\sqrt{\frac{4|\nu|}{\delta}}$. If
$\nu$ is negative, we have to change the orientation of the sequence.
The cyclic sequence is thus
\begin{equation}
\left\{ \left(\varepsilon\alpha,0\right),\left(0,\alpha\right),\left(-\varepsilon\alpha,0\right),\left(0,-\alpha\right)\right\} 
\end{equation}
where $\varepsilon=\text{sign}(\nu)$ changes the orientation of the
sequence (clockwise for $\varepsilon=1$, counterclockwise for $\varepsilon=-1$).

\subsection{Controller}

In this section, we propose a controller such that closed loop system
becomes 
\[
\dot{\mathbf{x}}=\mathbf{f}(\mathbf{x})\cdot a_{1}+\mathbf{g}(\mathbf{x})\cdot a_{2}+[\mathbf{f},\mathbf{g}]\left(\mathbf{x}\right)\cdot a_{3}
\]
where $\mathbf{a}=\left(a_{1},a_{2},a_{3}\right)$ is the new input
vector.

If we want to follow the field $a_{1}\mathbf{f}+a_{2}\mathbf{g}+a_{3}[\mathbf{f},\mathbf{g}]$,
we have to take the sequence
\begin{equation}
\left\{ \left(a_{1}+\varepsilon\alpha,a_{2}\right),\left(a_{1},a_{2}+\alpha\right),\left(a_{1}-\varepsilon\alpha,a_{2}\right),\left(a_{1},a_{2}-\alpha\right)\right\} 
\end{equation}
with $\alpha=\sqrt{\frac{4|a_{3}|}{\delta}}$ and $\varepsilon=\text{sign}(a_{3}).$
As illustrated by Figure \ref{fig:lie_multiplex}, we succeeded to
create a new input to our system which corresponds to the direction
pointed by the Lie bracket. Note that the equivalence given in the
figure should be understood as an approximation since $\delta$ is
not infinitesimal.

\begin{figure}[H]
\centering\includegraphics[width=14cm]{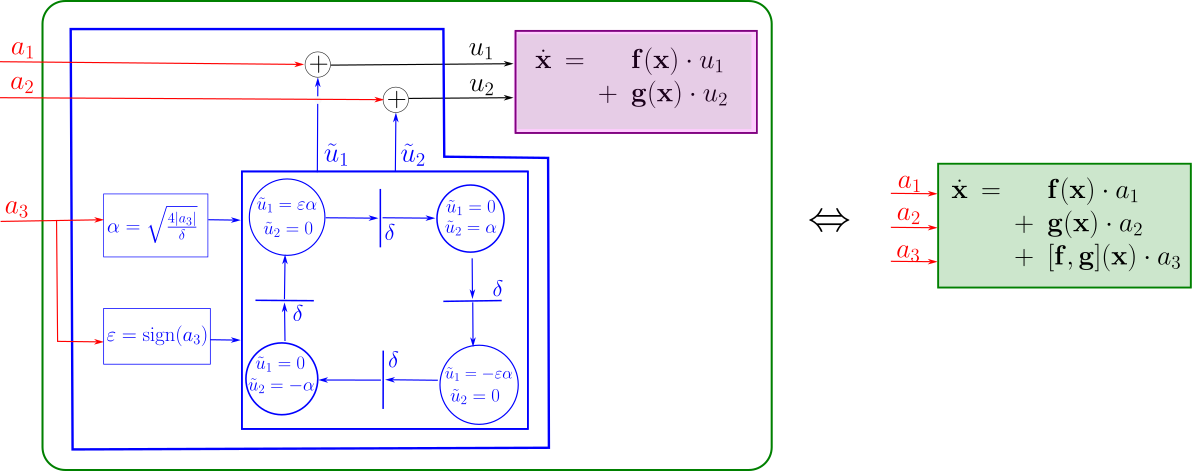}\caption{The feed forward controller (blue) allows us to control the system
in the direction of the Lie bracket}

\label{fig:lie_multiplex}
\end{figure}

\section{Control of the first order Dubins car\label{sec:first:order:dubins}}

\subsection{Dubins car}

Consider a Dubins car described by the following state equations:
\begin{equation}
\left\{ \begin{array}{lll}
\dot{x}_{1} & = & u_{1}\cos x_{3}\\
\dot{x}_{2} & = & u_{1}\sin x_{3}\\
\dot{x}_{3} & = & u_{2}
\end{array}\right.
\end{equation}
where $u_{1}$ is the speed of the cart, $u_{2}$ the rotation rate,
$x_{3}$ its orientation and $(x_{1},x_{2})$ the coordinates of its
center. Using Lie Brackets, we can add a new input to the system,
\emph{i.e.}, a new direction of control. We call this vehicle a first
order Dubins car, since we control the speed and the rotation rate
directly. It will be a second order Dubins car when their derivative
will be controlled instead.

\noindent If we set $\mathbf{x}=\left(x_{1},x_{2},x_{3}\right)$,
we have
\begin{equation}
\dot{\mathbf{x}}=\underset{\mathbf{f}(\mathbf{x})}{\underbrace{\left(\begin{array}{c}
\cos x_{3}\\
\sin x_{3}\\
0
\end{array}\right)}}\cdot u_{1}+\underset{\mathbf{g}(\mathbf{x})}{\underbrace{\left(\begin{array}{c}
0\\
0\\
1
\end{array}\right)}}\cdot u_{2}.
\end{equation}

\noindent We have
\begin{equation}
\begin{array}{ccl}
[\mathbf{f},\mathbf{g}]\left(\mathbf{x}\right) & = & \frac{d\mathbf{g}}{d\mathbf{x}}(\mathbf{x})\cdot\mathbf{f}(\mathbf{x})-\frac{d\mathbf{f}}{d\mathbf{x}}(\mathbf{x})\cdot\mathbf{g}(\mathbf{x})\\
 & = & \left(\begin{array}{ccc}
0 & 0 & 0\\
0 & 0 & 0\\
0 & 0 & 0
\end{array}\right)\left(\begin{array}{c}
\cos x_{3}\\
\sin x_{3}\\
0
\end{array}\right)-\left(\begin{array}{ccc}
0 & 0 & -\sin x_{3}\\
0 & 0 & \cos x_{3}\\
0 & 0 & 0
\end{array}\right)\left(\begin{array}{c}
0\\
0\\
1
\end{array}\right)\\
 & = & \left(\begin{array}{c}
\sin x_{3}\\
-\cos x_{3}\\
0
\end{array}\right)
\end{array}
\end{equation}

\noindent We conclude that we can now move the car laterally.

\subsection{Simulation}

We propose a simulation (taken from \citep{jaulinISTEroben} \citep{jaulin:inmooc})
to check the good behavior of your controller. To do this, take $\delta$
small enough to be consistent with the second order Taylor approximation
but large with respect to the sampling time $dt.$ For instance, we
may take $\delta=\sqrt{dt}.$ The initial state vector is taken as
$\mathbf{x}(0)=(0,0,1)$.

If we apply the cyclic sequence as a controller, we get 
\begin{equation}
\dot{\mathbf{x}}=\underset{\mathbf{f}(\mathbf{x})}{\underbrace{\left(\begin{array}{c}
\cos x_{3}\\
\sin x_{3}\\
0
\end{array}\right)}}\cdot a_{1}+\underset{\mathbf{g}(\mathbf{x})}{\underbrace{\left(\begin{array}{c}
0\\
0\\
1
\end{array}\right)}}\cdot a_{2}+\underset{[\mathbf{f},\mathbf{g}]\left(\mathbf{x}\right)}{\underbrace{\left(\begin{array}{c}
\sin x_{3}\\
-\cos x_{3}\\
0
\end{array}\right)}\cdot}a_{3}
\end{equation}

We made 4 simulations for $\mathbf{a}=(0.1,0,0)$, $\mathbf{a}=(0,0,0.1)$,
$\mathbf{a}=(-0.1,0,0)$ and $\mathbf{a}=(0,0,-0.1)$. We took for
initial vector $\mathbf{x}(0)=(0,0,1)$, $t\in[0,10]$ and $dt=0.01$.
We obtained the results depicted on Figure \ref{fig:lie_croix.pdf}.
We observe that after each simulation, the distance to the origin
is $0.1\cdot10=1$ approximately which is consistent we the fact that
$\mathbf{f}(\mathbf{x})$ and $[\mathbf{f},\mathbf{g}]$ have a norm
equal to 1. We did not show the simulation for $\mathbf{a}=(0,\pm0.1,0)$
since there no displacement: the car turns on itself.

\textcolor{magenta}{}
\begin{figure}[H]
\begin{centering}
\centering\includegraphics[width=12cm]{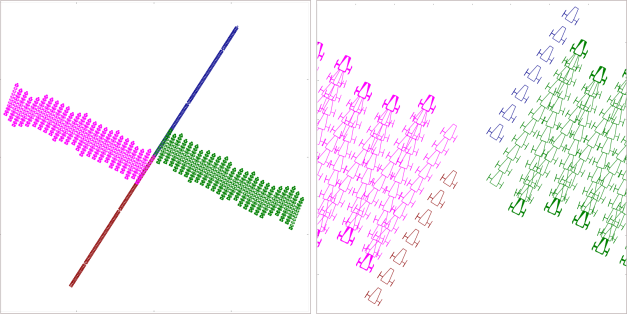}
\par\end{centering}
\caption{Left: simulation of the controller based on the Lie bracket technique.
The frame box is $[-1,1]\times[-1,1]$. Right: the same picture, but
with a frame box equal to $[-0.2,0.2]\times[-0.2,0.2]$. To avoid
superposition in the picture, the size of the car has been reduced
by $1/1000$. The length of the forward/backward subpaths corresponds
to 10cm approximately.}
\label{fig:lie_croix.pdf}
\end{figure}

\subsection{Control the cardinal directions}

We now propose a controller so that the car follows the cardinal directions,
\emph{i.e.}, North, South, East, West. We have

\begin{equation}
\begin{array}{ccc}
\dot{\mathbf{x}} & = & \mathbf{f}(\mathbf{x})\cdot a_{1}+\mathbf{g}(\mathbf{x})\cdot a_{2}+[\mathbf{f},\mathbf{g}]\left(\mathbf{x}\right)\cdot a_{3}\\
 & = & \underset{=\mathbf{A}\left(\mathbf{x}\right)}{\underbrace{\left(\begin{array}{ccc}
\cos x_{3} & 0 & \sin x_{3}\\
\sin x_{3} & 0 & -\cos x_{3}\\
0 & 1 & 0
\end{array}\right)}}\cdot\mathbf{a}
\end{array}
\end{equation}
We take $\mathbf{a}=\mathbf{A}^{-1}(\mathbf{x})\cdot\dot{\mathbf{x}}_{d}$
to get $\dot{\mathbf{x}}=\dot{\mathbf{x}}_{d}$, where $\dot{\mathbf{x}}_{d}=\left(\dot{x}_{1},\dot{x}_{2},\dot{x}_{3}\right)$.

For take sequentially the desired directions: $\dot{\mathbf{x}}_{d}=\left(1,0,0\right)$,
$\dot{\mathbf{x}}_{d}=\left(-1,0,0\right)$, $\dot{\mathbf{x}}_{d}=\left(0,-1,0\right)$,
$\dot{\mathbf{x}}_{d}=\left(0,1,0\right).$ We get the results shown
on Figure \ref{fig:lie_plus.pdf}.

\textcolor{magenta}{}
\begin{figure}[h]
\begin{centering}
\centering\includegraphics[width=12cm]{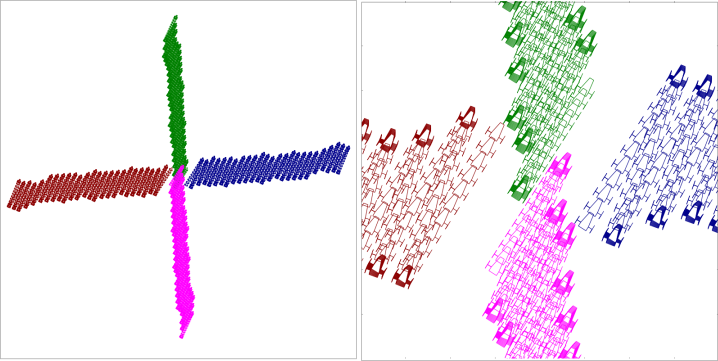}
\par\end{centering}
\caption{Left: The car goes from 0 toward all cardinal directions. The frame
box is $[-1,1]\times[-1,1]$. Right: the same picture, but with a
frame box equal to $[-0.2,0.2]\times[-0.2,0.2]$.}
\label{fig:lie_plus.pdf}
\end{figure}

Figure \ref{fig:lie_dubins1_inverse} shows the complete controller
for our Dubins car to follow a desired pose trajectory $\mathbf{p}(t)$.
For this, a high gain proportional controller was added. As we will
see in the following section, we have a right inverse of the first
order Dubins car.

\textcolor{magenta}{}
\begin{figure}[H]
\begin{centering}
\centering\includegraphics[width=10cm]{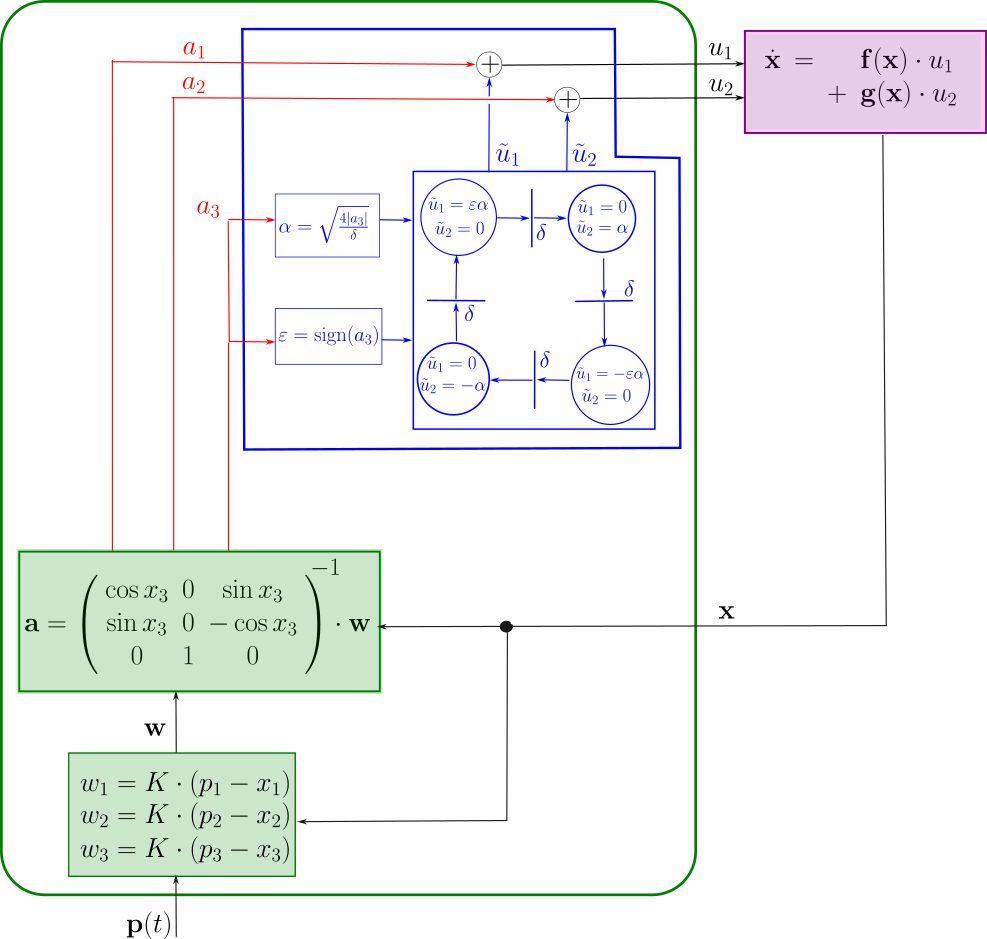}
\par\end{centering}
\caption{The complete controller for the Dubins car of order 1 }
\label{fig:lie_dubins1_inverse}
\end{figure}

\section{Control of the second order Dubins car \label{sec:second:order:dubins}}

\subsection{Model}

We consider the second order Dubins car given by

\begin{equation}
\left\{ \begin{array}{lll}
\dot{x}_{1} & = & x_{4}\cos x_{3}\\
\dot{x}_{2} & = & x_{4}\sin x_{3}\\
\dot{x}_{3} & = & x_{5}\\
\dot{x}_{4} & = & u_{1}\\
\dot{x}_{5} & = & u_{2}
\end{array}\right.
\end{equation}

or equivalently

\begin{equation}
\dot{\mathbf{x}}=\underset{\mathbf{f}(\mathbf{x})}{\underbrace{\left(\begin{array}{c}
x_{4}\cos x_{3}\\
x_{4}\sin x_{3}\\
x_{5}\\
0\\
0
\end{array}\right)}}+\underset{\mathbf{g}_{1}(\mathbf{x})}{\underbrace{\left(\begin{array}{c}
0\\
0\\
0\\
1\\
0
\end{array}\right)}}\cdot u_{1}+\underset{\mathbf{g}_{2}(\mathbf{x})}{\underbrace{\left(\begin{array}{c}
0\\
0\\
0\\
0\\
1
\end{array}\right)}}\cdot u_{2}
\end{equation}

The system has a drift $\mathbf{f}(\mathbf{x})$ which cannot be directly
controlled. To control such a system, we can use a backstepping technique
by decomposing the system as a sequence of right invertible systems. 

\subsection{Right inverse}

The notion of right invertibility comes from \citep{Hirschorn79}
and \citep{Doyle97}, Section 3.3.3. The notion right invertibility
is illustrated by Figure \ref{fig:rightinverse} and says that if
we put the controller $\mathcal{S}_{R}^{-1}$ in front of $\mathcal{S}$,
the effect of $\mathcal{S}$ will be canceled. Equivalently, we have
$\mathbf{y}=\mathcal{S}(\mathbf{u})=\mathcal{S}\circ\mathcal{S}_{R}^{-1}(\mathbf{v})\simeq\mathbf{v}$. 

\textcolor{magenta}{}
\begin{figure}[H]
\begin{centering}
\centering\includegraphics[width=12cm]{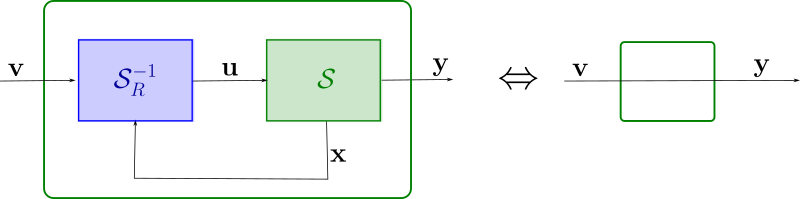}
\par\end{centering}
\caption{The system $\mathcal{S}_{R}^{-1}$ is the right inverse of $\mathcal{S}$}
\label{fig:rightinverse}
\end{figure}

Similarly, the notion left invertibility is illustrated by Figure
\ref{fig:leftinverse} and says that if we put the system $\mathcal{S}_{L}^{-1}$
after $\mathcal{S}$, the effect of $\mathcal{S}$ will be canceled,
\emph{i.e}, $\mathbf{v}=\mathcal{S}_{L}^{-1}\circ\mathcal{S}(\mathbf{u})\simeq\mathbf{u}$.
This left invertibility corresponds more or less to an observer and
will not be used later in this paper.

\textcolor{magenta}{}
\begin{figure}[H]
\begin{centering}
\centering\includegraphics[width=12cm]{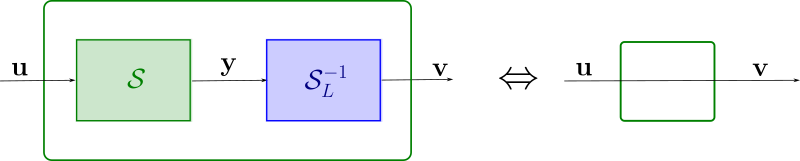}
\par\end{centering}
\caption{The system $\mathcal{S}_{L}^{-1}$ is the left inverse of $\mathcal{S}$}
\label{fig:leftinverse}
\end{figure}

\subsection{Backstepping control of the car}

As illustrated by Figure \ref{fig:lie_backstepping}, our Dubins car
can be decomposed into two systems. Both are right invertible. 
\begin{center}
\textcolor{magenta}{}
\begin{figure}[H]
\begin{centering}
\centering\includegraphics[width=10cm]{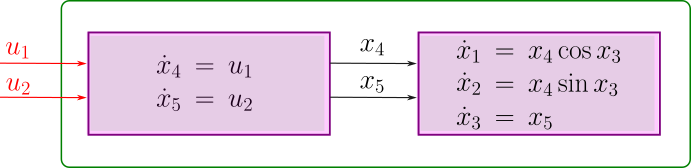}
\par\end{centering}
\caption{The second order Dubins car can be decomposed into a right invertible
system followed by a first order Dubins car}
\label{fig:lie_backstepping}
\end{figure}
\par\end{center}

Indeed, the first block of Figure \ref{fig:lie_backstepping_reg1},
corresponding to two integrators in parallel, can be eliminated by
adding a proportional controller with a gain $K$. If $K$ is large
enough, we can consider that $x_{4}=v_{1}$ and $x_{5}=v_{2}$. We
thus get a first order Dubins car. 
\begin{center}
\textcolor{magenta}{}
\begin{figure}[H]
\begin{centering}
\centering\includegraphics[width=12cm]{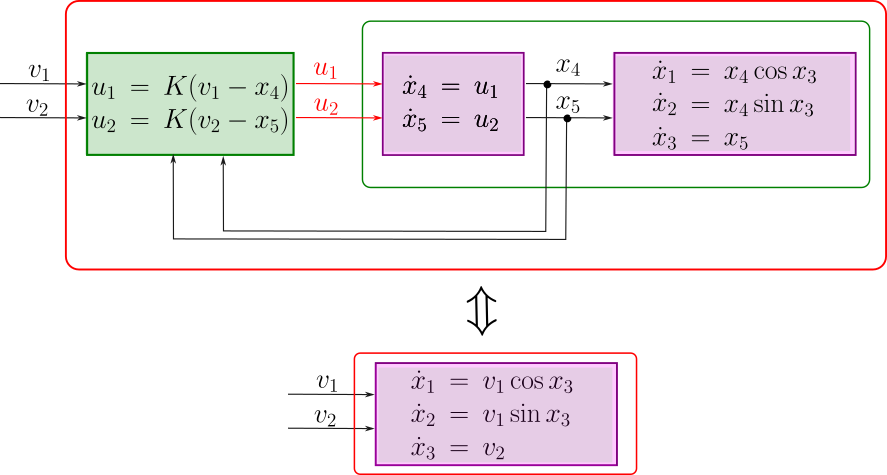}
\par\end{centering}
\caption{The first block of our Dubins car can be eliminated by high gain proportional
controller (green)}
\label{fig:lie_backstepping_reg1}
\end{figure}
\par\end{center}

As seen previously (see Figure \ref{fig:lie_multiplex}), the first
order Dubins, which is a driftless system with two inputs and three
states can be controlled using Lie brackets. As illustrated by Figure
\ref{fig:lie_backstepping_reg2}, the resulting vehicle can be controlled
tangentially (\emph{i.e.}, forward and backward), laterally (\emph{i.e.},
left and right), and in orientation, all independently.
\begin{center}
\textcolor{magenta}{}
\begin{figure}[H]
\begin{centering}
\centering\includegraphics[width=12cm]{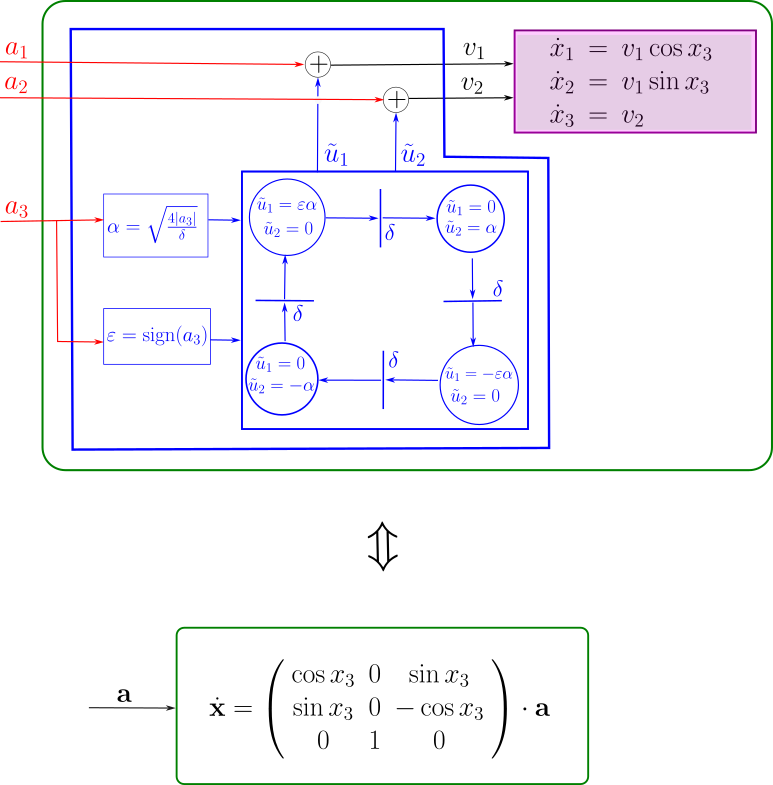}
\par\end{centering}
\caption{Adding a feed forward controller (blue) allows us to control all velocities,
in the car frame}
\label{fig:lie_backstepping_reg2}
\end{figure}
\par\end{center}

As illustrated by Figure \ref{fig:lie_backstepping_reg2}, using a
feedback linearization \citep{slotine91}, we can control the car
in all directions in the robot frame. 
\begin{center}
\textcolor{magenta}{}
\begin{figure}[H]
\begin{centering}
\centering\includegraphics[width=12cm]{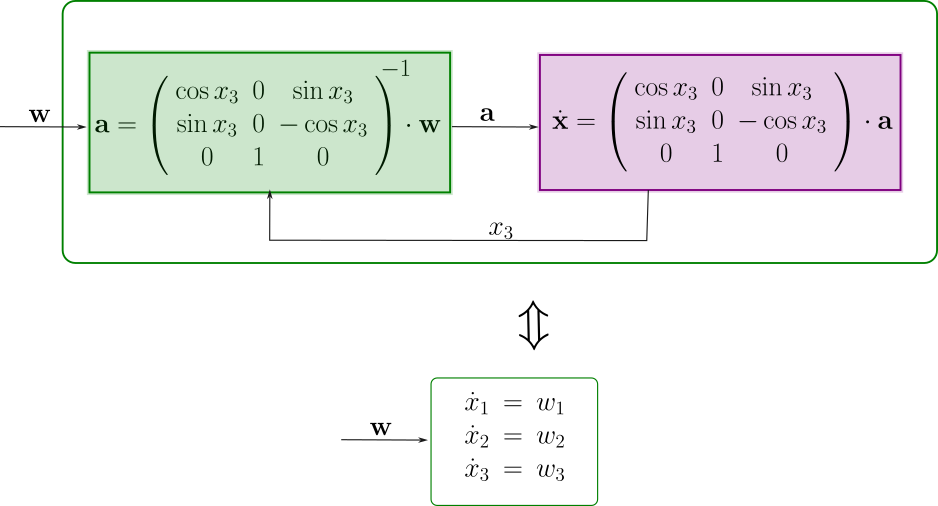}
\par\end{centering}
\caption{The green controller allows us to control the vehicle in all directions
(North, South, East, West) in the world frame and in orientation}
\label{fig:lie_backstepping_reg3}
\end{figure}
\par\end{center}

Finally, with a simple proportional control, we can follow any pose
trajectory $\mathbf{p}(t)$, (see Figure \ref{fig:lie_backstepping_reg4}),
\emph{i.e.}, a trajectory in the world frame, with a heading which
can be chosen independently. 
\begin{center}
\textcolor{magenta}{}
\begin{figure}[H]
\begin{centering}
\centering\includegraphics[width=7cm]{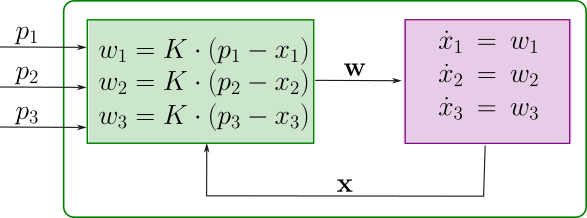}
\par\end{centering}
\caption{The last proportional controller (green) makes our second order Dubins
car follow a pose trajectory (position + heading) }
\label{fig:lie_backstepping_reg4}
\end{figure}
\par\end{center}

Figure \ref{fig:lie_backstepping_simu} shows the trajectory of the
car for the wanted Lissajou trajectory : 
\begin{equation}
\mathbf{p}(t)=\left(\begin{array}{c}
5\sin\frac{t}{100}\\
5\sin\frac{2t}{100}\\
0
\end{array}\right)
\end{equation}
and for different periods of time. Subfigure (a) shows the beginning
of the mission. We observe that the car maneuvers to have both a heading
to East and reach the target point (red). Subfigure (b) shows the
mission a few seconds later. The heading is now almost to East, as
expected. Subfigures (c) and (d) demonstrate the ability for the Dubins
to follows the Lissajou curve, always pointing to East. 

\textcolor{magenta}{}
\begin{figure}[h]
\begin{centering}
\centering\includegraphics[width=12cm]{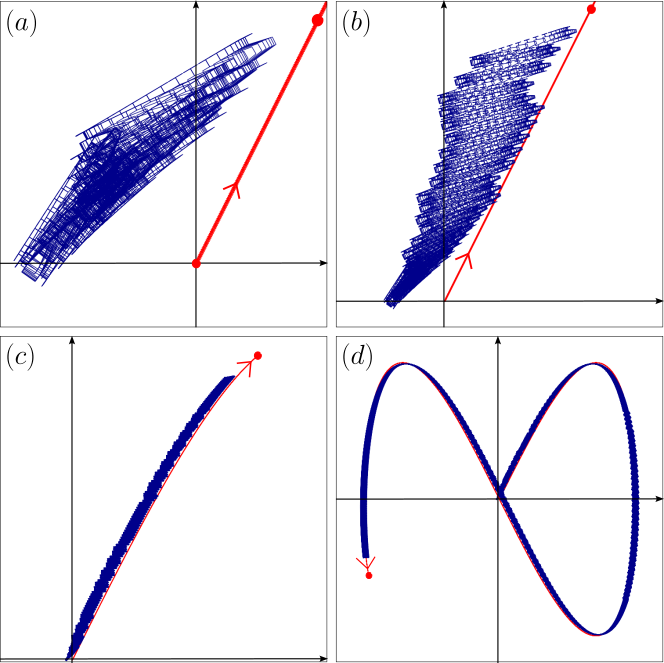}
\par\end{centering}
\caption{The simulation shows the Dubins car following a trajectory (here a
Lissajou curve in red), looking to the East }
\label{fig:lie_backstepping_simu}
\end{figure}

\section{Conclusion\label{sec:Conclusion}}

In control theory \citep{fantoni:02}\citep{Isidori95}\citep{KhalilHassa2002}\citep{LaValle:06},
when dealing with nonlinear systems, we generally ask to control $m$
states variables if we have $m$ inputs. Using Lie brackets, we can
control more than $m$ variables, but the new directions of control
only allow infinitesimal displacements in the state space. These displacements
can be obtained using infinitesimal cycles created by the Lie brackets.
These cycles may be interpreted as an infinitesimal time-division
multiplexing (TDM) method, used for transmitting and receiving independent
signals over a common signal path by means of synchronized switches. 

In this paper, we have illustrated the principle of the approach on
a first order and a second order Dubins cars. Using Lie brackets,
we have shown that we were able to move a car sideway to make a parking
slot, even if sideway motion are very slow.

The source codes of the simulations and a video associated to the
experiment are available at
\begin{center}
\href{https://www.ensta-bretagne.fr/jaulin/dubinsbrackets.html}{https://www.ensta-bretagne.fr/jaulin/dubinsbrackets.html}
\par\end{center}

\medskip{}

\bibliographystyle{plain}

\end{document}